\documentclass[12pt]{amsart}

\usepackage{amscd,amssymb,times, amsmath,color, fullpage}
\usepackage{graphicx}
\usepackage{enumitem}
\newtheorem{theorem}{Theorem}
\newtheorem{lemma}[theorem]{Lemma}
\newtheorem{proposition}[theorem]{Proposition}
\newtheorem{corollary}[theorem]{Corollary}
\theoremstyle{definition}

\newtheorem{example}[theorem]{Example}

\theoremstyle{remark}
\newtheorem{remark}[theorem]{Remark}

\begin{document}
\title{Invariant submanifolds of metric contact pairs}
\author{Amine Hadjar}
\address{Laboratoire de Math{\'e}matiques, Informatique et
Applications, Universit{\'e} de Haute Alsace - 4, Rue des
Fr{\`e}res Lumi{\`e}re, 68093 Mulhouse Cedex, France}
\email{mohamed.hadjar{\char'100}uha.fr}
\author{Paola Piu}
\address{Dipartimento di Matematica e Informatica, Universit\`a degli Studi di Cagliari, Via Ospedale 72, 09124 Cagliari, Italy}
\email{piu@unica.it}

\thanks{The second author was supported by a Visiting Professor fellowship at Universit\'e de Haute Alsace - Mulhouse in June 2014  and PRIN 2010/11 "Variet\`a reali e complesse: geometria, topologia e analisi armonica" - Italy; GNSAGA -INdAM, Italy}

\date{\today; MSC 2010 classification: primary 53C25; secondary 53B20, 53D10, 53B35, 53C12}

\maketitle
\vspace{5mm}

\begin{abstract}
We show that $\phi$-invariant submanifolds of metric contact pairs with orthogonal characteristic
foliations make  constant angles with the Reeb vector fields.
Our main result is that  for the normal case such submanifolds of dimension at least $2$ are
all minimal. We prove that an odd-dimensional $\phi$-invariant submanifold of a metric contact
pair with orthogonal characteristic foliations inherits a contact form with an almost contact metric
structure, and this induced structure is contact metric if and only if the submanifold is tangent to
one Reeb vector field and orthogonal to the other one.
Furthermore we show that the leaves of the two characteristic foliations of the differentials of the contact pair are minimal. 
We also prove that when one Reeb vector field is Killing and spans one characteristic foliation, the metric contact pair is a product of a contact metric manifold with $\mathbb{R}$.
\end{abstract}

\section{Introduction}
\noindent On a Riemannian manifold endowed with a tensor field  $\varphi$ of type $(1,1)$, a submanifold is said to be $\varphi$-invariant (or invariant) when its tangent bundle is preserved by $\varphi$.
Under some compatibility conditions between $\varphi$ and the metric, as is for example the case for almost Hermitian manifolds, the questions concerning the minimality of the submanifold and the nature of the induced structure are natural and interesting.
It is well known 
that invariant submanifolds of a contact metric manifold are minimal, and
the same holds for those of a K\"ahler manifold. 
However Vaisman  in \cite{Vaisman1982}
proved that a $J$-invariant submanifold of a Vaisman manifold ($J$ being the complex structure) is minimal if and only if it inherits a Vaisman structure. 
These manifolds are a subclass of locally conformally K\"ahler (lcK) manifolds.   
Dragomir and Ornea in \cite{Ornea} generalized this result by showing that a $J$-invariant submanifold of an lcK manifold is minimal if and only if the submanifold is tangent to the Lee vector field (and therefore tangent to the anti-Lee vector field).
By a statement of Bande and Kotschick  \cite{BK} normal metric contact pairs of type $(h,0)$  are nothing but non-K\"ahler Vaisman manifolds, so the results of Vaisman and Dragomir-Ornea apply to these manilfolds. 

A metric contact pair is a manifold endowed with a special case of a metric $f$-structure with two complemented frames in the sense of Yano \cite{yano}.
It carries an endomorphism field $\phi$ of corank $2$ and two natural commuting almost complex structures $J$ and $T$ of opposite orientations. 
When $J$ and $T$ are both integrable the structure is said to be normal. 
For $J$-invariant or $T$-invariant submanifolds on a normal metric contact pair with decomposable $\phi$ the problem was solved by Bande and the first author \cite{BH5} by proving that these submanifolds are minimal if and only if they are tangent to both the two Reeb vector fields of the structure, and they gave an example of such submanifold on which the contact pair of the ambient manifold does not induce a contact pair.
Regarding the $\phi$-invariant case they gave partial results.

In this paper we study $\phi$-invariant submanifolds of a metric contact pair $(M, \alpha_1, \alpha_2, \phi, g)$  of  any type $(h,k)$ with decomposable $\phi$, by  looking at the angles that Reeb vector fields make with these submanifolds.
We prove that these two angles are constant, then we solve completely the problem of minimality and that of induced structures.
First we show the following. 

A connected $\phi$-invariant submanifold $N$ of a metric contact pair with decomposable $\phi$ and Reeb vector fields $Z_1$ and $Z_2$, 
satisfies one of the following properties:
\begin{itemize}
\item[-] $N$ is even-dimensional and tangent to both $Z_1$ and $Z_2$.
\item[-] $N$ is $1$-dimensional and contained in one of the $2$-dimensional leaves of the vertical foliation spanned by $Z_1$ and $Z_2$.
\item[-] $N$ is of odd dimension $\geq 3$ everywhere tangent to one Reeb vector field $Z_1$ and orthogonal to the other one $Z_2$, or vice versa.
\item[-] $N$ is of odd dimension $\geq 3$, nowhere tangent and nowhere orthogonal to $Z_1$ and $Z_2$ making two constant angles with them.
\end{itemize}

For the minimality problem, as a consequence we prove the following.

Any $\phi$-invariant submanifold of dimension $\geq 2$ of a normal metric contact pair with decomposable $\phi$ is minimal. 
The $1$-dimensional case is obvious since it concerns vertical geodesics i.e. those which are integral curves of $c_1 Z_1+c_2 Z_2$ with $c_i$ constant functions.

Now let us return to the question concerning the induced structure on a $\phi$-invariant submanifold $N$ of a metric contact pair. 
When $N$ is even-dimensional, it is tangent to the Reeb vector fields and then it is $J$-invariant (and $T$-invariant). 
Although it does not always inherit a contact pair \cite{BH5}, we observe that it carries at least a metric $f$-structure with two complemented frames.
By normality on the ambient manifold we get a $\mathcal{K}$-structure (see \cite{BLY1973} for a definition) on $N$.
For the case $N$ of odd dimension along which the Reeb vector fields are nowhere tangent and nowhere orthogonal, when $\phi$ is decomposable, we prove that  
$N$ carries a contact form and an almost contact metric structure which is not contact metric, 
while for the other cases ($N$ tangent to one Reeb vector field and orthogonal to the other one) it has been shown in \cite{BLY} that $N$ is a contact metric submanifold of one of the characteristic leaves of the contact pair. By normality on the ambient manifold, the induced structure on $N$ will be normal.

Furthermore we prove that the leaves of the characteristic foliations of $d\alpha_1$ and $d\alpha_2$ of a metric contact pair $(M, \alpha_1, \alpha_2, \phi, g)$ with decomposable $\phi$ are minimal.  
Their leaves are $\phi$-invariant and tangent to both the two Reeb vector fields, but for our proof the integrability of $J$ or  $T$ is not needed.
We give an example where one of these two foliations is not totally geodesic.
When the type numbers of the contact pair are $(h,0)$ i.e. when the Reeb vector field $Z_2$ spans the characteristic distribution of $\alpha_1$, we prove the following.

If $Z_2$ is Killing, then the two characteristic foliations of the contact pair are totally geodesic, and then the metric contact pair is locally product of a contact metric manifold with $\mathbb{R}$. By normality the first factor will be Sasakian.

\section{Preliminaries}\label{preliminaries}
\noindent Blair, Ludden and Yano \cite{BLY} introduced in Hermitian geometry the notion of \emph{bicontact structure}.
This topic was formulated again as \emph{contact pair} by Bande in his PhD thesis in 2000, and then together with the first author in \cite{BH}.
A pair of $1$-forms $(\alpha_1, \alpha_2)$ on a manifold $M$ is said to be a \emph{contact pair} of type $(h,k)$ if
\begin{eqnarray*}
&\alpha_1\wedge (d\alpha_1)^{h}\wedge\alpha_2\wedge
(d\alpha_2)^{k} \qquad \text{is a volume form},\\
&(d\alpha_1)^{h+1}=0 \qquad  \text{and} \qquad (d\alpha_2)^{k+1}=0.
\end{eqnarray*}
The latter two conditions guarantee the integrability of the two 
subbundles of the tangent bundle
$$
T\mathcal{F}_i=\{X:\alpha_i(X)=0, d\alpha_i(X,Y)=0 \;  \forall Y   \}, i=1,2.
$$
They determine
the \emph{characteristic
  foliations} $\mathcal{F}_1$ of $\alpha_1$ and  $\mathcal{F}_2$ of $\alpha_2$ which are transverse and complementary. 
The leaves of $\mathcal{F}_1$ and $\mathcal{F}_2$ are contact manifolds of dimension $2k+1$ and $2h+1$ respectively, with contact forms induced by $\alpha_2$ and $\alpha_1$ (see \cite{BH}). 
We also define the $(2h+2k)$-dimensional \emph{horizontal subbundle} $\mathcal{H}$ to be the intersection of the kernels of $\alpha_1$ and $\alpha_2$. 

The equations
\begin{eqnarray*}
&\alpha_1 (Z_1)=\alpha_2 (Z_2)=1  , \qquad \; \alpha_1 (Z_2)=\alpha_2
(Z_1)=0 \, , \\
&i_{Z_1} d\alpha_1 =i_{Z_1} d\alpha_2 =i_{Z_2}d\alpha_1=i_{Z_2}
d\alpha_2=0 \, ,
\end{eqnarray*}
where $i_X$ is the contraction with the vector field $X$, determine uniquely the two vector fields $Z_1$ and $Z_2$, called \emph{Reeb vector fields}. Since they commute \cite{BH}, they give rise to a locally free $\mathbb{R}^2$-action, 
an integrable distribution called \emph{Reeb distribution},
and then a foliation $\mathcal{V}$ of $M$ by surfaces. The subbundle $T\mathcal{V}=\mathbb{R} Z_1 \oplus \mathbb{R} Z_2$ is called the \emph{vertical subbundle} and 
the tangent bundle of $M$
splits as:
$$
TM=T\mathcal F _1 \oplus T\mathcal F _2 =\mathcal{H}\oplus
\mathbb{R} Z_1 \oplus \mathbb{R} Z_2 .
$$

A \emph{contact pair structure} \cite{BH2} on a manifold $M$ is a triple
$(\alpha_1 , \alpha_2 , \phi)$, where $(\alpha_1 , \alpha_2)$ is a
contact pair and $\phi$ a tensor field of type $(1,1)$ such that:

\begin{eqnarray*}
\phi^2=-Id + \alpha_1 \otimes Z_1 + \alpha_2 \otimes Z_2 , \; \;
\phi Z_1=\phi Z_2 =0 .
\end{eqnarray*}
The rank of $\phi$ is
$\dim M -2$ and $\alpha_i \circ \phi =0$, for $i=1,2$.

The endomorphism $\phi$ is said to be \emph{decomposable}  if
$\phi (T\mathcal{F}_i) \subset T\mathcal{F}_i$, for $i=1,2$.
This is a natural condition that allows to have on each leaf of $\mathcal{F}_1$ and $\mathcal{F}_2$ an induced almost contact structure.

In \cite{BH3}  the notion of \emph{normality} for a contact pair structure is defined as the integrability of both of the two natural commuting almost complex structures of opposite orientations $J=\phi  - \alpha_2 \otimes Z_1 + \alpha_1 \otimes Z_2 $ and $T=\phi  + \alpha_2 \otimes Z_1 - \alpha_1 \otimes Z_2 $ on $M$.
This is equivalent to the vanishing of the tensor field
\begin{eqnarray*}
N^1 (X,Y)=  
[\phi , \phi ](X, Y) +2 d\alpha_1 (X,Y) Z_1 +2 d\alpha_2 (X,Y) Z_2 
\end{eqnarray*}
where $[\phi , \phi ]$ is the Nijenhuis tensor of $\phi$. 

A Riemannian metric $g$ on a manifold endowed with a contact pair structure is said to be \emph{associated}  
\cite{BH2}  if
$$g(X, \phi Y) = (d \alpha_1 + d \alpha_2) (X,Y)  \quad \text{and} \quad  g(X, Z_i) = \alpha_i(X)   \quad \text{for} \; i=1,2.
$$
 
Such a metric is necessarily \emph{compatible} with respect to the contact pair structure, which means that
$$g(\phi X,\phi Y)=g(X,Y)-\alpha_1 (X)
\alpha_1 (Y)-\alpha_2 (X) \alpha_2 (Y).$$
Moreover the subbundles $\mathbb{R}Z_1$, $\mathbb{R}Z_1$ and $\mathcal{H}$ are pairwise orthogonal.

A \emph{metric contact pair} on a manifold $M$ is a
four-tuple $(\alpha_1, \alpha_2, \phi, g)$ where $(\alpha_1,
\alpha_2, \phi)$ is a contact pair structure and $g$ an associated
metric with respect to it. Such a manifold $M$ will also be called a metric contact pair for short.

On a metric contact pair the endomorphism field $\phi$ is decomposable if and only if the characteristic foliations $\mathcal{F}_1$, 
$\mathcal{F}_2$ are orthogonal 
\cite{BH2}.
In this case the leaves of $\mathcal{F}_j$ are minimal submanifolds \cite{BH4},  they carry contact metric structures induced by 
$( \phi, Z_i, \alpha_i, g)$, for $j\neq i$ (see \cite{BH2}), and 
by the normality they become Sasakian \cite{BH3}.
Of course the product of two contact metric manifolds (or a contact metric manifold with $\mathbb{R}$) gives rise to a metric contact pair with decomposable endomorphism, and the structure is normal if and only if the two factors are Sasakian.
It is important to note that there exist metric contact pairs with decomposable $\phi$ which are not locally products of contact metric manifolds as shown in the following example (see also \cite{BH4} for a similar construction).

\begin{example}\label{liegroup}
Consider the simply connected $6$-dimensional nilpotent Lie group $G^6$ with structure
equations:
\[
d \alpha_1 = \alpha_3 \wedge \alpha_5 \quad, \quad  d \alpha_2 = \alpha_4 \wedge \alpha_6
\quad, \quad d \alpha_6 = \alpha_4 \wedge \alpha_5,
 \]
\[
d \alpha_3 = d \alpha_4 = d \alpha_5 = 0
\]
The pair $(\alpha_1,\alpha_2)$ is a contact pair of type $(1,1)$ with Reeb vector fields $(Y_1,Y_2)$, the $Y_i$'s being dual to the $\alpha_i$'s.
The characteristic distribution of $\alpha_1$ (respectively $\alpha_2$) is spanned by $Y_4$, $Y_6$ and $Y_2$ (respectively $Y_3$, $Y_5$ and $Y_1$).
Take the metric
\[
g = \alpha_1 ^2 + \alpha_2 ^2 + \frac{1}{2}\left( \alpha_3 ^2 + \alpha_4 ^2 + \alpha_5 ^2 +\alpha_6 ^2 \right)
\]
and the decomposable endomorphism $\phi$ defined to be zero on $Y_1$, $Y_2$, and
\[
\phi Y_5 = Y_3 \quad, \quad  \phi Y_3 = - Y_5  \quad, \quad
\phi Y_6 = Y_4 \quad, \quad  \phi Y_4 = - Y_6
\]
Then $(\alpha_1,\alpha_2, \phi, g)$ is a left invariant metric contact pair on the Lie group $G^6$.
We can easily see that the leaves of the two characteristic foliations are Sasakian, though the metric contact pair is not normal because $N^1 (Y_3,Y_4)=[\phi,\phi](Y_3,Y_4) = Y_4 \neq 0$.
We also remark that the characteristic foliation of $\alpha_1$ is totally geodesic, while from 
$g\left(\nabla_{Y_4} Y_6, Y_5\right) = - \frac{1}{2}$
 follows that the characteristic foliation of $\alpha_2$ is not totally geodesic.
 
Since the structure constants of the nilpotent Lie algebra of $G^6$ are rational, there exist cocompact lattices $\Gamma$ of $G^6$. Now the metric contact pair on $G^6$ descends to all these quotients $G^6/\Gamma$ and we obtain closed nilmanifods carrying the same type of structure. All the remarks concerning the structure we constructed on $G^6$ still remain valid on the metric contact pairs $G^6/\Gamma$.
\end{example}

For the normal case we can give the following.

\begin{example}\label{H6}
As a manifold consider the product $H^6=\mathbb{H}_3 \times \mathbb{H}_3$ where $\mathbb{H}_3$ is the $3$-dimensional Heisenberg group. Let $\left\{ \alpha_1, \alpha_2, \alpha_3 \right \}$ (respectively $\left \{ \beta_1,  \beta_2,  \beta_3 \right\}$) be a basis of the cotangent space at the identity for the first (respectively second) factor $\mathbb{H}_3$ satisfying
$$
d\alpha_3=\alpha_1\wedge \alpha_2  \, , \qquad  d\alpha_1 = d\alpha_2 = 0,
$$
$$
d\beta_3=\beta_1\wedge \beta_2 \, , \qquad d\beta_1=d\beta_2=0.$$
The pair $(\alpha_3 , \beta_3 )$ determines a contact pair of type $(1,1)$ on $H^6$ with Reeb vector fields $(X_3 , Y_3 )$, the $X_i$'s (respectively the $Y_i$'s) being dual to the $\alpha_i$'s (respectively the $\beta_i$'s). The left invariant metric 
$$
g=\alpha_3^2 + \beta_3^2 +\frac {1}{2}(\alpha_1^2 + \beta_1^2 +\alpha_2^2 + \beta_2^2 )
$$
is associated to the pair with decomposable endomorphism $\phi$ given by $\phi(X_2)=X_1$ and $\phi(Y_2)=Y_1$.
The metric contact pair $(H^6,\alpha_3, \beta_3, \phi, g)$ is normal because it is the product of two Sasakian manifolds.
Also here $H^6$ admits cocompact lattices $\Gamma$ and the structure descends to the nilmanifolds $H^6/\Gamma$  as normal metric contact pairs.
\end{example}

Some other
interesting examples and properties of such structures were given in 
\cite{BBH,BH,BH2,BH3,BH4,BH5}.
In the following remark we describe how metric contact pairs relate to other well-known structures.

\begin{remark}
Normal metric contact pairs with decomposable endomorphism were already studied in \cite{BLY} under the name bicontact Hermitian manifolds of bidegree $(1,1)$.
They were regarded as a generalization of the Calabi-Eckmann manifolds.
A metric contact pair of type $(h,k)$ is a special case of metric $f$-structure of rank $2h+2k$ with two complemented frames in the sense of Yano \cite{yano}.
The normality condition of metric contact pairs i.e. the integrability of both almost complex structures $J$ and $T$ is equivalent to the normality condition as an $f$-structure which consists exactly on the vanishing of  the tensor field $N^1$ described before.
We can also observe that a normal metric contact pair is a special case of $\mathcal{K}$-structures in the sense of Blair, Ludden and Yano \cite{BLY1973}.
It has been shown in \cite{BK} that normal metric contact pairs of type $(h,0)$ are nothing but non-K\"ahler Vaisman manifolds (called $\mathcal{PK}$-manifolds in \cite{Vaisman}). The $\mathcal{P}$-manifolds of Vaisman \cite{Vaisman} are necessarily metric contact pairs of type $(h,0)$ where the Reeb vector fields are Killing, and they include the subclass of $\mathcal{PK}$-manifolds.
\end{remark}

In the course of our work we will need the following lemmas.

\begin{lemma}[Bande et al. \cite{BBH}]\label{nablaXZi}
On a manifold endowed with a contact pair structure with decomposable $\phi$ and a compatible metric, for every $X$ we have that $\nabla_X Z_1$ and $\nabla_X Z_2$ are horizontal.
\end{lemma}

\begin{lemma}\label{nablaZiX}
On a manifold endowed with a contact pair structure with decomposable $\phi$ and a compatible metric,
for every $X$ horizontal we have that $\nabla_ {Z_1}X$ and $\nabla_{ Z_2}X$ are horizontal.
\end{lemma}
\begin{proof}
By the previous lemma, $\alpha_j(\nabla_X Z_i)=0$ for $i,j=1,2$, giving
\begin{eqnarray*}
\alpha_j(\nabla_ {Z_i}X)
&=& \alpha_j([Z_i,X])   \\
&=&  Z_i \alpha_j(X)-X \alpha_j(Z_i)-2d\alpha_j(Z_i,X) \\
&=&0.
\end{eqnarray*}
\end{proof}

\section{$\phi$-invariant submanifolds}\label{section-phi-invariant}

\noindent A submanifold $N$ of a metric contact pair is said to be $\phi$-invariant if its tangent bundle $TN$ is preserved by the endomorphism field $\phi$.
We will denote by $Z_i^T$ (respectively, $Z_i^\bot$) the tangential (respectively, normal) component of the two Reeb vector fields $Z_1$ and $Z_2$ along $N$.
In the following proposition we recall some properties from \cite{BH5} we need and concerning the positions of the Reeb vector fields along a $\phi$-invariant submanifold $N$.

\begin{proposition}[Bande et al. \cite{BH5}]\label{the4areverticalnobothorthogonal}~
\begin{enumerate}
\item \label{the4arevertical} Along the $\phi$-invariant submanifold $N$ the four sections $Z_1^T$, $Z_2^T$, $Z_1^\bot$ and $Z_2^\bot$ are vertical.
\item \label{nobothorthogonal}There is no point $p$ of $N$ such that the tangent vectors $(Z_1)_{p}$ and $(Z_2)_p$ are both orthogonal to the tangent space $T_pN$.
\item \label{orthogonaltangent} If at a point $p$ of $N$ one Reeb vector field is tangent to $N$ and the second one is transverse, then the second one is orthogonal to $N$ at $p$.

\end{enumerate}
\end{proposition}

To this we can add the following.

\begin{proposition}\label{oneorthogonalonetangent}
If at a point $p$ of the $\phi$-invariant submanifold $N$ one of the Reeb vector fields is orthogonal to $N$, then the second one is tangent to $N$ at $p$.
\end{proposition}

\begin{proof}
Suppose that $(Z_1)_p$ is orthogonal to $T_pN$.
First by Proposition \ref{the4areverticalnobothorthogonal}-(\ref{nobothorthogonal}) we have $(Z_2^T)_p\neq 0$.
Moreover $(Z_2^T)_p$ is orthogonal to $(Z_1)_p$.
Next by Proposition \ref{the4areverticalnobothorthogonal}-(\ref{the4arevertical}) the vector $(Z_2^T)_p$ lies in the plane $\left((Z_1)_{p},(Z_2)_p\right)$ and we get $(Z_2^T)_p=(Z_2)_p$.
\end{proof}

Now after these observations we can state the following proposition.

\begin{proposition}\label{proposition2caspairimpair}
A $\phi$-invariant submanifold which is tangent to both the Reeb vector fields has even dimension. Otherwise, its dimension is odd and its tangent bundle intersects the Reeb distribution along a line bundle.
\end{proposition}

Indeed, it is clear that when a submanifold is tangent to both $Z_1$ and $Z_2$ at a point, since its tangent space at that point is preserved by $\phi$ it is also preserved by $J$ which is almost complex. Thus the dimension is even.

Now let us describe a little more a $\phi$-invariant submanifold $N$ which is not tangent to the Reeb distribution.
We have seen that the four sections $Z_1^T$, $Z_1^\bot$, $Z_2^T$ and $Z_2^\bot$ are vertical.
When at a point $p$ of $N$ one of the two Reeb vector fields is not tangent to $N$, then the second one is not orthogonal to the submanifold   $N$ (Proposition \ref{oneorthogonalonetangent}), and this occurs on a whole open set of $N$. Moreover at those points 
both families $\{Z_1^T, Z_2^T\}$ and  $\{Z_1^\bot,Z_2^\bot\}$ have rank one.
Any tangent vector of $N$ at $p$ which is orthogonal to $\{Z_1^T, Z_2^T\}$ is horizontal, since it is orthogonal to $\{Z_1^\bot,Z_2^\bot\}$ too and then orthogonal to $\{Z_1,Z_2\}$. Then by the $\phi$-invariance of $N$, $\phi$ is almost complex when acting on horizontal vectors tangent to $N$, i.e. on the \emph{horizontal part} $\mathcal{H}\cap TN$ of $TN$.
This explains the odd dimension of $N$ and the fact that the Reeb distribution still remains  nowhere tangent to the submanifold.
Notice that $\{Z_1^T, Z_2^T\}$ spans a line bundle which is the intersection of the tangent bundle $TN$ with the Reeb distribution that can be called the \emph{vertical part} of $TN$.
Moreover we have 
$$
TN= \left(T\mathcal{V} \cap TN \right) \oplus \left(\mathcal{H} \cap TN \right).
$$

When $N$ is a $1$-dimensional $\phi$-invariant submanifold, then it is any $1$-dimensional submanifold of any $2$-dimensional leaf of the vertical foliation $\mathcal{V}$. So in the sequel we will suppose that the dimension of $N$ is at least $3$.

For simplicity suppose that $Z_2$ is nowhere tangent to $N$, so that $Z_1$ is nowhere orthogonal to $N$ (by Proposition \ref{oneorthogonalonetangent}). Normalizing $Z_1^T$ we get on $N$ a unit vector field spanning the vertical part $T\mathcal{V}\cap TN$ of $TN$:
$$
\zeta=\frac{1}{\Vert Z_1^T\Vert}Z_1^T.
$$
Now along $N$  the equation
$$
\zeta= (\cos \theta_1) Z_1+(\sin \theta_1) Z_2
$$
defines a smooth function $\theta_1$ on $N$ taking values in $]-\pi/2,\pi/2[$.
This function is the measure of  the angle  that $Z_1$ makes with the submanifold, i.e. the oriented angle $(Z_1,Z_1^T)$  in the plane 
$(Z_1,Z_2)$ oriented by the almost complex structure $J$. 
When $Z_1$ is nowhere tangent to the submanifold, similarly we get a function $\theta_2$ measuring the oriented angle $(Z_2,Z_2^T)$ that is the angle between $Z_2$ and the submanifold, and satisfying 
$Z_2^T/{\Vert Z_2^T\Vert}=(\cos \theta_2) Z_2+(\sin \theta_2) (-Z_1)$ with $-\pi/2<\theta_2<\pi/2.$

We will say that the submanifold $N$ is \emph{leaning} when both Reeb vector fields are leaning along $N$, i.e. when they are nowhere tangent and nowhere orthogonal to $N$. This means that the functions $\theta_1$ and $\theta_2$  are well defined and take nonvanishing values.

\begin{theorem}\label{theoremleaning}
Let $(M, \alpha_1, \alpha_2, \phi, g)$ be a metric contact pair of type $(h,k)$ with decomposable $\phi$. Suppose that $M$ carries a leaning $\phi$-invariant submanifold $N$ of odd dimension $2n+1\geq 3$. Then,
\begin{enumerate} 
\item \label{inegalite} for the angle $(Z_1,Z_1^T)$ we have $0<\theta_1<\pi/2$,
\item for the angle $(Z_2,Z_2^T)$ we have $\theta_2=\theta_1-\pi/2$,
\item \label{anglevecteurhorizontal} each nonzero horizontal tangent vector $X$ of $N$ decomposes as $X=X_1+X_2$, $X_i$ being nonzero horizontal vector tangent to  the characteristic foliation $\mathcal{F}_j$ for $j\neq i$, with $\Vert X_2\Vert^2/ \Vert X_1\Vert^2= \tan \theta_1$, 
\item for the type numbers we have $h\geq n$ and $k\geq n$, i.e. the dimensions of the two characteristic foliations are at least $2n+1$.
\end{enumerate}
\end{theorem}

In this theorem Property \ref{anglevecteurhorizontal} means that $X$ makes with $T\mathcal{F}_2$ (or more precisely with $\mathcal{H}\cap T\mathcal{F}_2$) an angle of absolute value $\theta^\prime _1=\arctan\sqrt{\tan\theta_1}$, and of course with $\mathcal{H}\cap T\mathcal{F}_1$ an angle $\theta^\prime _2=\pi /2 -\theta^\prime _1$.
Property \ref{inegalite} states that along $N$ the vertical part of $TN$ separates the vertical plane $(Z_1,Z_2)$ into two half-planes, each one containing one Reeb vector field, and for $\zeta$ we also have
$$
\zeta=\frac{1}{\Vert Z_2^T\Vert}Z_2^T=(\cos \theta_2) Z_2+(\sin \theta_2) (-Z_1).
$$ 

\begin{proof}
Take any local horizontal nonvanishing vector field $X$ of $N$. 
Such a vector field always exists since the dimension of $N$ is at least $3$.
Along $N$ it decomposes as $X=X_1+X_2$ with $X_i$ horizontal and tangent to $\mathcal{F}_j$, for $j\neq i$. 
By the $\phi$-invariance of $N$, $\phi X$ is also tangent to $N$, and $\phi X_i$ is horizontal and tangent to $\mathcal{F}_j$, for $j\neq i$, by decomposability of $\phi$. In order to compute the vertical part of $[X,\phi X]$ we have
$$
g([X,\phi X],Z_1) 
=\alpha_1([X_1,\phi X_1])+\alpha_1([X_1,\phi X_2])+\alpha_1([X_2,\phi X_1])+\alpha_1([X_2,\phi X_2]).
$$
For the first term we have $\alpha_1([X_1,\phi X_1])=-2d\alpha_1(X_1,\phi X_1)=2g(X_1,X_1)$, and the last three terms vanish. Thus we have
$$
g([X,\phi X],Z_1)=2 \Vert X_1\Vert ^2
$$
and similarly 
$$
g([X,\phi X],Z_2)=2 \Vert X_2\Vert ^2
$$
so that the vertical part of $[X,\phi X]$ is the vector field $2 \Vert X_1\Vert ^2Z_1+2 \Vert X_2\Vert ^2Z_2$.

Since $[X,\phi X]$ is also tangent to $N$, its vertical part is then collinear with $\zeta= (\cos \theta_1) Z_1+(\sin \theta_1) Z_2$.
Now from $\cos\theta_1 >0$ and $\sin \theta_1 \neq 0$, we obtain that $\sin \theta_1 >0$ i.e. $0<\theta_1<\pi/2$, and $\theta_2=\theta_1-\pi/2$.
Moreover $X_1$, $X_2$ do not vanish. For the measure of the angle $(X,X_1)$ we obtain its tangent which is $\Vert X_2\Vert / \Vert X_1\Vert=\sqrt{ \tan \theta_1}$.

Regarding the dimensions of the characteristic foliations, take at any point $p$ of $N$ a basis $\{e_1, \dots, e_{2n}  \}$ of the horizontal part $\mathcal{H}\cap T_pN$ of the tangent space $T_pN$.
Each vector $e_l$ being horizontal decomposes as $e_l=e_{l1}+e_{l2}$ with $e_{li}\in \mathcal{H}\cap T_p\mathcal{F}_j$ for $j\neq i$.
Let $\lambda_1, \dots, \lambda_{2n}  $ any real numbers such that $\sum_{l=1}^{2n}\lambda_l e_{l1}=0$.
Put $X=\sum_{l=1}^{2n}\lambda_l e_{l}$.
Then $X=X_1+X_2$ with $X_i=\sum_{l=1}^{2n}\lambda_l e_{li}$ lying in $ T_p\mathcal{F}_j$ for $j\neq i$.
Applying Property \ref{anglevecteurhorizontal} of this theorem to the horizontal vector $X$ we get $X=0$ since $X_1$ is supposed to be zero.
Hence from $\sum_{l=1}^{2n}\lambda_l e_{l}=0$ we obtain that $\lambda_l=0$ for all $l$.
Finally the vectors $e_{11}, \dots, e_{2n\,1} $ are linearly independent in $\mathcal{H}\cap T_p\mathcal{F}_2$ and then $h\geq n$.
In the same way we get $k\geq n$, and this completes the proof.
\end{proof}

\begin{example}

Consider the metric contact pairs on the nilpotent Lie group $G^6$ and its closed nilmanifolds $G^6/\Gamma$ described in Example \ref{liegroup}. For any two arbitrary nonzero real numbers $a$ and $b$, the three vectors
$$
X = a Y_3 + b Y_6  , \quad \phi X = - a Y_5 + b Y_4 \quad\text{and}\quad Z =[X,\phi X] = a^2 Y_1 + b^2 Y_2
$$
span a $\phi$-invariant subalgebra of the Lie algebra  of $G^6$ which determines a $3$-dimensional foliation $\mathcal{N}$ in $G^6$ (and also in the nilmanifolds $G^6/\Gamma$). Each leaf $N$ of $\mathcal{N}$ is $\phi$-invariant, leaning, minimal and non totally geodesic. The vertical part of $TN$ is spanned by $Z$ restricted to $N$, and the angle $\theta_1$ that the Reeb vector field $Y_1$ makes with $N$ satisfies
$$
\cos\theta_1=a^2/\sqrt{a^4+b^4} \qquad \text{and} \qquad \sin\theta_1=b^2/\sqrt{a^4+b^4}.
$$
By a suitable choice of $a$ and $b$, we can see that the angle $\theta_1$ can take any value in $]0,\pi/2[$.
\end{example}

In the same way, we can have $\phi$-invariant submanifolds on the normal metric contact pairs $H^6$ and its nilmanifolds described in Example \ref{H6}.

\begin{example}
Take again any nonzero real numbers $a$ and $b$.
The three vectors $X=a X_2 + b Y_1 $, $\phi X=a X_1 -b Y_2$ and $Z=[X,\phi X]= a^2 X_3 + b^2 Y_3$ span a $\phi$-invariant subalgebra of the Lie algebra 
of $H^6$ which determines a $3$-dimensional foliation $\mathcal{N}$ in $H^6$ (and also in each nilmanifold $H^6/\Gamma$).
Each leaf is $\phi$-invariant, leaning and totally geodesic.
Moreover the angle  that the Reeb vector field $X_3$ makes with the leaf has tangent equal to $b^2/a^2$.

In order to obtain an example of a closed $\phi$-invariant leaning submanifold, we have just to choose suitably a lattice of $H^6$.
Indeed let   $L_e$ be the leaf passing through the identity element of  the Lie group $H^6$.
We can see that the Lie subgroup $L_e$ is nothing but the Heisenberg group that admits cocompact lattices (see e.g. \cite{BH5} or \cite{GW} to get an explicit one). Take any of such lattices which we will call $\Gamma$. 
Because $L_e$ is a subgoup of $H^6$, we have that $\Gamma$ is also a lattice of $H^6$.
Now the closed nilmanifold $N^3=L_e/\Gamma$ of $L_e$ is a submanifold of the nilmanifold $M^6= H^6 /\Gamma$ of $H^6$. 
As explained before since the normal contact pair on 
$H^6$ is left invariant, it descends to the quotient $M^6$ as a normal metric contact pair
$(\tilde{\alpha_3}, \tilde{\beta_3}, \tilde{\phi}, \tilde{g})$ 
of type $(1,1)$ with decomposable endomorphism $\tilde{\phi}$. 
Finally we obtain a normal metric contact pair $M^6$ with a decomposable endomorphism, carrying a closed leaning $\tilde{\phi}$-invariant submanifold $N^3$. 
\end{example}

Before stating our main theorem concerning minimality of $\phi$-invariant submanifolds on normal metric contact pairs, 
we first give an important lemma which allows us to describe more the angles between odd-dimensional $\phi$-invariant submanifold and the Reeb vector fields.

\begin{lemma}\label{lemmeconstant}
Let $(M, \alpha_1, \alpha_2, \phi, g)$ be a metric contact pair with decomposable $\phi$, and $N$ a connected $\phi$-invariant submanifold of odd dimension $\geq 3$. 
If at a point of $N$ the Reeb vector field $Z_2$ is not tangent to $N$, then $Z_2$ is everywhere transverse to $N$,  the measure $\theta_1$ of the oriented angle that $Z_1$ makes with $N$
is constant and $0\leq\theta_1<\pi/2$.
\end{lemma}

Actually for the leaning case the four angles $\theta_1$, $\theta_2$, $\theta^\prime_1$ and $\theta^\prime_2$ described above are all constant. 
For $\theta_1=0$, at each point $Z_1$ is tangent and $Z_2$ is orthogonal to the submanifold. 
As a matter of fact $Z_2$ must be transverse to $N$ because of the odd dimension and then orthogonal to $N$ by Proposition \ref{the4areverticalnobothorthogonal}-(\ref{orthogonaltangent}).
For example this is the case for the leaves of the charateristic foliation $\mathcal{F}_2$. 
\begin{proof}
The proof will be given in two stages.
Let $N^\prime$ be the nonempty open set of points of $N$ on which $Z_2$ is transverse to $N$. The function $\theta_1$ is well defined on $N^\prime$ and we have obviously $0\leq\theta_1<\pi/2$.
We will first prove that $\theta_1$  is locally constant on $N^\prime$ (Stage 1). Next we will prove that $N^\prime$ is nothing but $N$ (Stage 2).

\vspace{0,2 cm}
\noindent Stage 1, step 1:\newline
Let $X$ be any (local) unit vector field of $N^\prime$ orthogonal to $\zeta$.
Then it is horizontal because it is also orthogonal to $J\zeta$. Since $\zeta$ is unit $g(\nabla_X \zeta,\zeta)=0$.
We also have
$$
g(\nabla_X \zeta,J\zeta)=g(\nabla_{\zeta}X,J\zeta)+g([X,\zeta],J\zeta)=0
$$
since $\nabla_{\zeta}X$ is horizontal by Lemma \ref{nablaZiX}, and $[X,\zeta]$ is tangent to $N^\prime$.
Because $\zeta$ and $J\zeta$ span the vertical bundle, $\nabla_X \zeta$ is horizontal along $N^\prime$.

\vspace{0,2 cm}
\noindent Stage 1, step 2:\newline
Differentiating $\zeta= (\cos \theta_1) Z_1+(\sin \theta_1) Z_2$, we obtain
\begin{equation}\label{equationnablaXzeta}
\nabla_X \zeta=X(\theta_1)J\zeta + \cos \theta_1 \nabla_X Z_1 +\sin\theta_1 \nabla_X Z_2.
\end{equation}
By step 1 and Lemma \ref{nablaXZi} we have that $\nabla_X \zeta$, $\nabla_X Z_1$ and $\nabla_X Z_2$ are horizontal, then 
\eqref{equationnablaXzeta} implies that along $N^\prime$
$$
X(\theta_1)=0.
$$

\vspace{0,2 cm}
\noindent Stage 1, step 3:\newline
Take an $X$ as above. The dimension of $N^\prime$ being $\geq 3$ such an $X$ always exists.
By the $\phi$-invariance of $N^\prime$, $\phi X$ and $[X,\phi X]$ are also tangent to the submanifold. 
Using the fact that $X$ and $\phi X$ are horizontal, we have 
$$
g(Z_1+Z_2, [X,\phi X])=(\alpha_1+\alpha_2)([X,\phi X])=-2(d\alpha_1+d\alpha_2)(X,\phi X)=2g(X,X)=2. 
$$
Replacing $Z_1+Z_2$ by $(\cos \theta_1+\sin \theta_1)\zeta+(\cos \theta_1-\sin \theta_1)J\zeta$ and using $0\leq\theta_1<\pi/2$ we obtain
$$
g([X,\phi X], \zeta)=\frac{2}{ \cos \theta_1+\sin \theta_1}
$$
because $J\zeta$ is orthogonal to the submanifold. Then we have
\begin{equation}\label{crochetXfiX}
[X,\phi X]=\frac{2}{ \cos \theta_1+\sin \theta_1}\zeta +Y
\end{equation}
for some $Y$ horizontal and tangent to the submanifold. By step 1 since $X$, $\phi X$ and $Y$ are horizontal and tangent to $N^\prime$, we get $X(\theta_1)=\phi X(\theta_1)=0$ then $[X,\phi X](\theta_1)=0$, and also $Y(\theta_1)=0$.
Hence \eqref{crochetXfiX} implies that along $N^\prime$
$$
\zeta	(\theta_1)=0.
$$

Finally by steps 2 and 3 the function $\theta_1$ is locally constant on the open set $N^\prime$.

\vspace{0,2 cm}
\noindent Stage 2:\newline
 Let $N^{\prime\prime}$ be the complement set of the open set $N^\prime$ in $N$, that is the set of the points of $N$ where $Z_2$ is tangent to $N$.
 At these points $Z_1$ is not tangent to $N$, and the function $\theta_2$ is well defined on an open set of $N$ containing $N^{\prime\prime}$. By the same arguments as before, $\theta_2$ is locally constant, giving that $N^{\prime\prime}$ is open because it consists on the vanishing points of $\theta_2$.
 By the connectedness of $N$ we have that $N^{\prime\prime}$ is empty and
 $$N=N^\prime.$$

Now the function $\theta_1$ is well defined and constant on the whole $N$, and of course $Z_2$ is everywhere transverse to the submanifold, completing the proof.
\end{proof}

A first immediate consequence of our Lemma \ref{lemmeconstant}, using Proposition \ref{proposition2caspairimpair}, is the following  theorem which describes all possible relative positions of a $\phi$-invariant submanifold with respect to both the Reeb vector fields.

\begin{theorem}\label{theoremclassification}
Let $N$ be a connected $\phi$-invariant submanifold of a metric contact pair with decomposable $\phi$.  
Then $N$ satisfies one of the following properties:
\begin{enumerate}
\item \label{caseeven} $N$ is even-dimensional and tangent to both Reeb vector fields.
\item \label{casedim1}  $N$ is $1$-dimensional and contained in one of the $2$-dimensional leaves of the vertical foliation.
\item \label{caseodd1tangent} $N$ is of odd dimension $\geq 3$ everywhere tangent to one Reeb vector field $Z_1$ and orthogonal to the other one $Z_2$, or vice versa.
\item \label{caseoddleaning} $N$ is of odd dimension $\geq 3$, nowhere tangent and nowhere orthogonal to the Reeb vector fields making two constant angles with them.
\end{enumerate}
\end{theorem}

\begin{remark}
Except the $1$-dimensional case, a $\phi$-invariant submanifold of a metric contact pair with decomposable $\phi$ always makes a constant angle with each of the two Reeb vector fields.
\end{remark}
\section{Minimality}  

\noindent We now turn to our main result on minimality of $\phi$-invariant submanifolds of normal metric contact pairs with orthogonal characteristic foliations.

\begin{theorem}\label{corollaireinvariantisminimal}
Any $\phi$-invariant submanifold of dimension $\geq 2$ of a normal metric contact pair with decomposable $\phi$ is minimal. 
\end{theorem}

\begin{proof}
Consider a connected $\phi$-invariant submanifold $N$  
of a normal metric contact pair with decomposable $\phi$. 
When the dimension of $N$ is $\geq 2$, $N$ satisfies one of the cases (\ref{caseeven}), (\ref{caseodd1tangent}) or (\ref{caseoddleaning})
enumerated in Theorem \ref{theoremclassification}.
Take then the question case-by-case and use some partial results from \cite{BH5} to conclude.
\noindent
Assume that $N$ has even dimension (Case \ref{caseeven}). Then it is tangent to both the Reeb vector fields and one can readily show that $N$ is also $J$-invariant. Moreover the normality of the metric contact pair implies the integrability of $J$. 
By \cite{BH5}, when $J$ is integrable a $J$-invariant submanifold is minimal if and only if it is tangent to the Reeb distribution. So this applies to $N$ and then $N$ is minimal.

Another result from \cite{BH5} states that, on a normal metric contact pair with decomposable $\phi$, a $\phi$-invariant submanifold tangent to one Reeb vector field and orthogonal to the other one (Case \ref{caseodd1tangent}) is minimal.

For the very remaining possible case (Case 4) i.e. when $N$ is leaning, we use the following.
On a normal metric contact pair with decomposable $\phi$, 
a $\phi$-invariant leaning submanifold is minimal if and only if the angle between one Reeb vector field, say $Z_1$, and its tangential  part $Z_1^T$ is constant along the line curves of $Z_1^T$ (see \cite{BH5}). Now by Lemma \ref{lemmeconstant}, for our case this angle is constant on the whole $N$, and then $N$ is minimal.
\end{proof}

Observe that by Proposition \ref{proposition2caspairimpair}, a connected $1$-dimensional submanifold is $\phi$-invariant if and only if it is contained in one leaf of the vertical foliation $\mathcal{V}$. This foliation is totally geodesic and the geodesics are integral curves of nonzero vertical vector fields $c_1Z_1+c_2 Z_2$ with $c_i$ constant functions, since we have $\nabla_{Z_i}Z_j=0$ for $i,j=1,2$ (see \cite{BH2}). 
Hence for the Case \ref{casedim1} of Theorem \ref{theoremclassification} we have

\begin{theorem}
A connected $1$-dimensional $\phi$-invariant submanifold of a metric contact pair is minimal if and only if it is tangent to a vector field of the form $c_1Z_1+c_2Z_2$ with $c_i$ real numbers.
\end{theorem}

This means that minimal $1$-dimensional $\phi$-invariant submanifolds are exactly vertical geodesics.

\section{Induced structures on $\phi$-invariant submanifolds}

\noindent In general, an even-dimensional $\phi$-invariant submanifold of a metric contact pair does not inherit necessarily a contact pair structure \cite{BH5}. Anyway we can observe that it still carries some interesting structure.

\begin{proposition}
On a metric contact pair  a $\phi$-invariant submanifold $N$ tangent to both Reeb vector fields carries a metric $f$-structure with two complemented frames. A normal metric contact pair on the ambient manifolds induces a $\mathcal{K}$-structure on $N$.
\end{proposition}
Here the notion of $f$-structure is meant in the sense of Yano \cite{yano} with $f=\phi$ restricted to $N$, and the complemented frames are
the restrictions of the Reeb vector fields to $N$.
For the definition of $\mathcal{K}$-structure see \cite{BLY1973}.
The proof of this proposition is a straightforward computation and will be
omitted.

For the odd-dimensional case, when $\phi$ is decomposable, a $\phi$-invariant submanifold tangent to one Reeb vector field and orthogonal to the other one inherits a contact metric structure, and it is Sasakian when the metric contact pair is normal  \cite{BLY}.   
The following statement concerns the remaining class of $\phi$-invariant submanifolds, i.e. those which are leaning.

\begin{theorem}\label{inducedstructure}
Let $(M, \alpha_1, \alpha_2, \phi, g)$ be a metric contact pair with decomposable $\phi$, and $N$ a $\phi$-invariant submanifold of $M$ of dimension $\geq3$ nowhere tangent and nowhere orthogonal to the Reeb vector fields $Z_1$ and $Z_2$.
Set $
\zeta=\frac{1}{\Vert Z_1^T\Vert}Z_1^T
$ and $\omega=g(\zeta,\cdot )$ along $N$, where $Z_1^T$ is the tangential part of $Z_1$.
Then 
\begin{enumerate}
\item $\omega$ induces  
a contact form on $N$ with Reeb vector field $\zeta$,
\item \label{compatible} $( \phi ,   \zeta, \omega, g)$ induces  
an almost contact metric structure on $N$, which is not contact metric.
\end{enumerate}
If the metric contact pair is normal, the induced almost contact structure on $N$ is normal.
\end{theorem}

For an almost contact manifold $(N,\omega, \zeta, \phi)$, we use the terminology of \cite{Blairbook}. 
A metric is said to be  {\it associated} when $g(X,\zeta)=\omega(X)$ and $g(X,\phi Y)= d\omega(X,Y)$ for all $X,Y$, while {\it compatible} means just that $g(\phi X,\phi Y)=g(X,Y)-\omega(X)\omega(Y)$ for all $X,Y$.
By the conclusion (\ref{compatible}) in our Theorem \ref{inducedstructure}, the induced metric is  compatible but it is not associated. Actually at each point of $N$ we have $g(X,\phi Y)\neq d\omega(X,Y)$ for some $X,Y$ tangent to $N$ at the same point.

\begin{proof}
As we know $
\zeta= (\cos \theta_1) Z_1+(\sin \theta_1) Z_2
$ for some constant function $\theta_1$ with $0<\theta_1<\pi/2$.
This gives  along $N$
$$
\omega= (\cos \theta_1) \alpha_1+(\sin \theta_1) \alpha_2.
$$
The $1$-form $\omega$ induces a contact form on $N$ when for all $X$ tangent to $N$, $\omega(X)=0$ and $i_X d\omega=0$ imply that $X=0$. Take any $X$ tangent to $N$ such that $\omega(X)=0$, then $X$ is orthogonal to $\zeta$ so that $X$ is horizontal. Put $X=X_1+X_2$ with $X_i\in \mathcal{H}\cap T\mathcal{F}_j$, for $j\neq i$. If in addition $i_X d\omega=0$ we have
\begin{eqnarray*}
0 &=& - d\omega(X, \phi X) = - (\cos \theta_1)  d\alpha_1(X, \phi X) - (\sin \theta_1)  d\alpha_2(X, \phi X) \\
&=&- (\cos \theta_1)  d\alpha_1(X_1, \phi X_1) - (\sin \theta_1)  d\alpha_2(X_2, \phi X_2) \\
&=&- (\cos \theta_1) ( d\alpha_1 + d\alpha_2)(X_1, \phi X_1) - (\sin \theta_1)  ( d\alpha_1 + d\alpha_2)(X_2, \phi X_2)  \\
&=& (\cos \theta_1) \Vert X_1 \Vert^2 + (\sin \theta_1) \Vert X_2 \Vert^2 . 
\end{eqnarray*}
Now since $\cos \theta_1 > 0 $ and $\sin\theta_1 >0$, we get
$
X_1 = X_2 = 0
$
 and then $X = 0$ as desired.
 Hence $\omega$ induces a contact form on $N$.
 
 It is clear that $\omega(\zeta)=1$. We also have $i_{\zeta}d\omega=0$ because $\zeta$ is vertical and then lying in the kernels of $d\alpha_1$ and $d\alpha_2$. Hence $\zeta$ is the Reeb vector field of the contact form on $N$.

To prove that $ \phi^2 = - I + \omega \otimes \zeta $, first remark that 
$\phi\zeta = 0$ and $\omega(\zeta)=1$. Next any $X$ tangent to $N$ and orthogonal to $\zeta$ is horizontal and then satisfies $\phi^2 X=-X=-X+ \omega(X)\zeta$. Hence $( \phi ,   \zeta, \omega)$ induces an almost contact structure on $N$.

To get the compatibility of the metric with the almost contact structure, we need to prove that $g(\phi X,\phi Y) = g(X,Y) - \omega(X) \omega(Y)$ for every  $X$ and $Y$ tangent to $N$ . Such a vector $X$ decomposes as $X=\omega(X)\zeta + X_H$ for some horizontal vector $X_H$ tangent to $N$, and the same for $Y$.
Replacing in $g(\phi X,\phi Y) = g(X,Y) - \alpha_1(X) \alpha_2(Y) - \alpha_1(Y) \alpha_2(X)$  we obtain
\begin{eqnarray*}
g(\phi X,\phi Y) 
&=& g(X,Y) - \alpha_1(\omega(X)\zeta) \alpha_2(\omega(Y)\zeta) - \alpha_1(\omega(Y)\zeta) \alpha_2(\omega(X)\zeta) \\
&=& g(X,Y) - \omega(X)\omega(Y) \left(\alpha_1^2(\zeta) +\alpha_2^2(\zeta)\right) \\
&=& g(X,Y) - \omega(X)\omega(Y) 
\end{eqnarray*}
Hence $( \phi ,   \zeta, \omega, g)$ induces an almost contact metric structure on $N$. 

Now we prove that the induced metric is not associated. 
Since $N$ is of dimension $\geq3$, at every point of $N$ there exists a nonzero vector $X$ tangent to $N$ and orthogonal to $\zeta$. By Theorem \ref{theoremleaning}, $X=X_1+X_2$ with $X_i$ nonzero horizontal and tangent to $\mathcal{F}_j$, for 
$j\neq i$.
Choosing $Y=\phi X$ let us compare $g(X,\phi Y)$ with $d\omega(X,Y)$. 
On the one hand $-g(X,\phi Y)=g(X,X)=\Vert X_1\Vert ^2 +\Vert X_2\Vert ^2$, and on the other hand as before we have
\begin{eqnarray*}
-d\omega(X,Y)
&=& -d\omega(X,\phi X) \\
&=& (\cos \theta_1) \Vert X_1 \Vert^2 + (\sin \theta_1) \Vert X_2 \Vert^2 . 
\end{eqnarray*}
The difference being 
$(1-\cos \theta_1) \Vert X_1 \Vert^2 + (1-\sin \theta_1) \Vert X_2 \Vert^2 >0 $, 
we get $g(X,\phi Y)\neq d\omega(X,Y)$. Hence the metric is not associated for the almost contact structure of $N$.

When the contact pair structure $(\alpha_1,\alpha_2,\phi)$ is normal, for every $X$ and $Y$ tangent to $N$ we have
$$
[\phi,\phi](X,Y)+2d\alpha_1(X,Y)Z_1+2d\alpha_2(X,Y)Z_2=0
$$
and for its orthogonal projection on $N$ using the $\phi$-invariance of $N$ we get
\begin{eqnarray*}
0&=&[\phi,\phi](X,Y)+2(\cos \theta_1)d\alpha_1(X,Y)\zeta+2\sin(\theta_1)d\alpha_2(X,Y)\zeta\\
&=& [\phi,\phi](X,Y)+2d\omega(X,Y)\zeta,
\end{eqnarray*}
which means that the almost contact structure on $N$ is normal. This completes the proof.
\end{proof}

We can summarize our discussion above concerning the induced structures on $\phi$-invariant manifolds as follows.
\begin{theorem}
Let $(M, \alpha_1, \alpha_2, \phi, g)$ be a metric contact pair with decomposable $\phi$ and Reeb vector fields $Z_1$ and $Z_2$, and $N$ a connected $\phi$-invariant submanifold of $M$ of dimension $\geq 2$.
 
\begin{enumerate}
\item When $N$ has even dimension, then it inherits a metric $f$-structure with two complemented frames.
\item Suppose that $N$ has odd-dimension. Along $N$ let $\zeta$ be the normalized vector field of a nonzero vector field among $Z_1^T$ and $Z_2^T$ (the tangential parts of $Z_1$ and $Z_2$), and set $\omega=g(\zeta,\cdot )$. Then
\begin{enumerate}
\item $\omega$ induces  
a contact form on $N$ with Reeb vector field $\zeta$, and $( \phi ,   \zeta, \omega, g)$ induces  
an almost contact metric structure on $N$.
\item the induced almost contact metric on $N$ is contact metric if and only if  $N$ is tangent to one Reeb vector field and orthogonal to the other one.
\end{enumerate}
\end{enumerate}
In all cases, if  
the metric contact pair on $M$ is normal, the induced structure on $N$ is normal.
\end{theorem}

\section{Characteristic leaves of $d\alpha_1$ and $d\alpha_2$}

\noindent Consider a metric contact pair $(M, \alpha_1, \alpha_2, \phi, g)$ of type $(h,k)$ with decomposable $\phi$ and Reeb vector fields $Z_1$ and $Z_2$. 
In \cite{BH2},
it has been shown that the $2$-dimensional vertical foliation $\mathcal{V}$ tangent to the Reeb distribution is totally geodesic.
Actually, this is even true in general for any compatible metric $g$ with respect to a contact pair structure $(\alpha_1, \alpha_2, \phi)$ without decomposability condition on $\phi$.
 
It is also known  that the characteristic foliations $\mathcal{F}_1$ and $\mathcal{F}_2$ of the $1$-forms
$\alpha_1$ and $\alpha_2$ respectively are orthogonal, and their leaves are minimal  \cite{BH4}.

Now recall the existence of two other remarkable foliations in the manifold $M$, which are the characteristic foliations $\mathcal{G}_1$ and $\mathcal{G}_2$ of the $2$-forms $d\alpha_1$ and $d\alpha_2$ respectively.
Their corresponding subbundles are
$$
T\mathcal{G}_i=\{X: d\alpha_i(X,Y)=0 \;  \forall Y   \}, \qquad i=1,2.
$$
Each leaf of $\mathcal{G}_1$ (respectively $\mathcal{G}_2$) inherits a metric contact pair of type $(0,k)$ (respectively $(h,0)$), and is foliated by leaves of $\mathcal{V}$ and also by leaves of $\mathcal{F}_1$ (respectively $\mathcal{F}_2$)  \cite{BH2}.
Moreover the leaves of $\mathcal{G}_1$ and $\mathcal{G}_2$ are $\phi$-invariant and tangent to both $Z_1$ and $Z_2$, and  we have the following minimality theorem.

\begin{theorem}\label{ch4-th-Gi-minimalfoliations}
On a metric contact pair $(M, \alpha_1, \alpha_2, \phi, g)$ with decomposable $\phi$, the leaves of the characteristic foliations $\mathcal{G}_1$ and $\mathcal{G}_2$ of the $2$-forms $d\alpha_1$ and $d\alpha_2$  are minimal.
\end{theorem}
By normality condition on the metric contact pair of the ambient manifold $M$, it has been shown in \cite{BH5} that the leaves of $\mathcal{G}_1$ and $\mathcal{G}_2$ are minimal. In our theorem, the normality condition is not needed.
\begin{proof}
To prove the minimality of the leaves of $\mathcal{G}_1$ in the Riemannian manifold $(M,g)$, we use the minimality criterion of Rummler \cite{Rummler}.
If the type numbers of the contact pair are $(h,k)$,
the dimension of this foliation is $2k+2$.
The volume element of the metric $g$ can be written as \cite{BH4} 
$$
dV= \frac{(-1)^{h+k}}{2^{h+k} h! k!} \alpha_1 \wedge (d\alpha_1) ^h \wedge \alpha_2 \wedge (d\alpha_2) ^k,
$$
 so that the characteristic $(2k+2)$-form of the foliation $\mathcal{G}_1$  is, up to a constant,
$
\omega = \alpha_1 \wedge \alpha_2 \wedge (d\alpha_2) ^k .
$ 
Since the Reeb vector field $Z_1$ is tangent to the foliation $\mathcal{G}_1$ and 
 $$i_{Z_1} d\omega=i_{Z_1} (d \alpha_1 \wedge \alpha_2 \wedge (d\alpha_2) ^k) =0,$$
 the characteristic $(2k+2)$-form of the foliation $\mathcal{G}_1$ is closed on the subbundle $T\mathcal{G}_1$, giving that the leaves of $\mathcal{G}_1$ are minimal.
 The same argument applies to $\mathcal{G}_2$, completing the proof.
\end{proof}

\begin{example}
In Example \ref{liegroup} the characteristic subbundle $T\mathcal{G}_1$ of $d\omega_1$ (respectively $T\mathcal{G}_2$ of $d\omega_2$) is spanned by $Y_1$, $Y_2$, $Y_4$ and $Y_6$ (respectively $Y_1$, $Y_2$, $Y_3$ and $Y_5$).
Observe that the leaves of  $\mathcal{G}_2$ are totally geodesic while those of $\mathcal{G}_1$ are minimal but not totally geodesic.
\end{example}

\section{Metric contact pairs of type $(h,0)$}
\noindent 
On a metric contact pair with decomposable endomorphism $\phi$ the leaves of the two characteristic foliations are $\phi$-invariant submanifolds.
They are minimal and a priori they are not totally geodesic \cite{BH4}.
However when the contact pair is of type $(h,0)$ we can state the following.

\begin{theorem}\label{theoremKillingproduct}
Consider a metric contact pair $(M, \alpha_1, \alpha_2, \phi, g)$ of type $(h,0)$ and Reeb vector fields $Z_1$ and $Z_2$.
If $Z_2$ is Killing, then the metric contact pair $M$ is locally the product of a contact metric manifold with $\mathbb{R}$.
\end{theorem}

We observed that $\mathcal{P}$-manifolds of Vaisman are metric contact pairs of type $(h,0)$ where Reeb vector fields are Killing.
Theorem \ref{theoremKillingproduct} is a generalization to metric contact pairs of a result stated for $\mathcal{P}$-manifolds by Vaisman \cite{Vaisman}.

\begin{proof}
The two characteristic foliations are orthogonal and complementary. The leaves of $\mathcal{F}_1$ are 
the integral curves of $Z_2$ which are geodesics \cite{BH2}.
We have just to prove that the leaves of $\mathcal{F}_2$ are totally geodesic.
Any leaf $F$ of $\mathcal{F}_2$ is a  submanifold of codimension one, 
and the vector field $Z_2$ restricted to $F$ is the normal to the leaf.
When $Z_2$ is Killing, 
any geodesic $\gamma$ of $(M,g)$ starting from a point $p$ of $F$ and tangent to $F$ at $p$ satisfies 
$$\dot{\gamma}\,g\left(\dot{\gamma},Z_2\right) = g\left(\nabla_{\dot{\gamma}}\;\dot{\gamma},Z_2\right)  +  g\left(\dot{\gamma},\nabla_{\dot{\gamma}}\;Z_2\right) = 0.$$
Therefore $g(\dot{\gamma},Z_2)=0$ because $\dot{\gamma}$ and $Z_2$ are orthogonal at $p$.
Hence the geodesic $\gamma$ remains in the leaf $F$.
\end{proof}

We know that the normality of a metric contact pair implies that the Reeb vector fields are Killing \cite{BBH}. 
The following corollary is an immediate consequence of our previous theorem, and it has been already stated in equivalent terms of $\mathcal{PK}$-manifolds (nowadays called non-K\"ahler Vaisman manifolds) by Vaisman.
\begin{corollary}[Vaisman \cite{Vaisman}]
A normal metric contact pair of type $(h,0)$ is locally the product of a Sasakian manifold with $\mathbb{R}$.
\end{corollary}

\end{document}